\documentclass[11pt]{article}
\usepackage[T1]{fontenc}
\usepackage{amsfonts}
\usepackage{amsmath}
\usepackage{amssymb}
\usepackage{amsthm}
\usepackage{bbm}
\usepackage{bm}
\usepackage{mathrsfs}
\usepackage{verbatim}
\usepackage{setspace}
\usepackage{color}
\usepackage{pdfsync}
\usepackage{enumitem}
\usepackage{mathtools}
\usepackage{nicefrac}
\theoremstyle{plain}
\newtheorem{theorem}{Theorem}[section]
\newtheorem{proposition}[theorem]{Proposition}
\newtheorem{lemma}[theorem]{Lemma}

\theoremstyle{definition}
\newtheorem{definition}[theorem]{Definition}
\newtheorem{remark}[theorem]{Remark}

\newtheorem{example}[theorem]{Example}

\newtheorem{assumption}[theorem]{Assumption}

\theoremstyle{remark}

\renewenvironment{thebibliography}[1]{%
\begin{oldthebibliography}{#1}%
\setlength{\baselineskip}{.9em}
\linespread{1}
\small
\setlength{\parskip}{0.3ex}%
\setlength{\itemsep}{.5em}%
}%
{%
\end{oldthebibliography}%
}

\newcommand{\eps}{\varepsilon}

\newcommand{\F}{\mathbb{F}}

\newcommand{\N}{\mathbb{N}}

\newcommand{\R}{\mathbb{R}}

\newcommand{\Z}{\mathbb{Z}}

\newcommand{\cB}{\mathcal{B}}

\newcommand{\cD}{\mathcal{D}}

\newcommand{\cF}{\mathcal{F}}

\newcommand{\cH}{\mathcal{H}}

\newcommand{\cK}{\mathcal{K}}

\newcommand{\cP}{\mathcal{P}}

\newcommand{\cT}{\mathcal{T}}

\newcommand{\1}{\mathbbm{1}}

\newcommand{\fP}{\mathfrak{P}}

\newcommand{\fM}{\mathfrak{M}}

\DeclareMathOperator{\dom}{dom}

\DeclareMathOperator{\NA}{NA}

\def\hypo{\mathrm{hypo}\,}

\numberwithin{equation}{section}

\usepackage[pdfborder={0 0 0}]{hyperref}
\hypersetup{
  urlcolor = black,
  pdfauthor = {Ariel Neufeld, Mario \v{S}iki\'c},
  pdfkeywords = {Nonconcave Robust Optimization; Knightian uncertainty
},
  pdftitle = {Nonconcave Robust Optimization},
  pdfsubject = {Nonconcave Robust Optimization},
  pdfpagemode = UseNone
}

\begin{document}

\title{\vspace{-4em}
Nonconcave Robust Optimization with Discrete Strategies under Knightian Uncertainty 
\date{\today}
\author{
  Ariel Neufeld%
  \thanks{
  Division of Mathematical Sciences, Nanyang Technological University, Singapore, \texttt{ariel.neufeld@ntu.edu.sg}.
  Financial support by the NAP Grant as well as ETH RiskLab and the Swiss National Foundation Grant SNF 200020$\_$172815 is gratefully acknowledged.
  }
  \and
  Mario \v{S}iki\'c
  \thanks{
    Center for Finance and Insurance, University of Zurich, \texttt{mario.sikic@bf.uzh.ch}.
  }
 }
}
\maketitle \vspace{-1.2em}

\begin{abstract}
We study robust stochastic optimization problems in the quasi-sure setting in discrete-time. The strategies in the multi-period-case are restricted to those taking values in a discrete set. The optimization problems under consideration are not concave.  We provide conditions under which a maximizer exists. The class of problems covered by our robust optimization problem includes optimal stopping and semi-static trading under Knightian uncertainty. 
\end{abstract}

\vspace{.9em}

{\small

}


\section{Introduction}\label{sec:intro}
We consider the following robust stochastic optimization problem
\begin{equation}\label{eq:robust-optim-Intro}
\sup_{H\in\cH}\,\inf_{P\in\fP} E^P\big[\Psi(H)\big],
\end{equation}
where $\cH$ is a set of stochastic processes and represent the control variable, e.g. the portfolio evolution of a trader; $\fP$ is a set of probability measures  modeling uncertainty. The goal of this paper is to establish existence of a maximizer $\widehat H$ for a class of maps $\Psi$.

Problem~\eqref{eq:robust-optim-Intro} is called a robust optimization problem, since it asks for the best possible performance in the worst possible situation, modeled by a probability measure $P\in\fP$. We will provide an example of the set $\fP$ in Section~\ref{sec:examples}; see also \cite{Bartl.16}.

The main focus of the paper is to analyze the situation when the map $\Psi$ is not necessarily concave. Concave functions enjoy a number of properties that make them easy to work with. One important property that makes them amenable to the techniques developed in robust finance, see e.g.~\cite{BouchardNutz.13, BertsekasShreve.78}, is that they are locally Lipschitz continuous if their domain has nonempty interior. This allows one to be able to evaluate the function value just by knowing it on a countable set of points, which is independent of the function under consideration. This property, that it is enough to know the function on a countable number of points, proved to be crucial in the analysis of the problem; see~\cite{NeufeldSikic.16}. Although this approximation result remains valid also for nonconcave functions, see e.g.~\cite{sikic-nonar}, in general it depends on the function. For that reason we impose this property on our maps $\Psi$ by assuming a sort of discreteness property. 

Although our formulation of the problem is quite general, we will provide examples showing that nonconcavity of market models arises naturally in finance. Examples include
semi-static trading with integer positions, optimal stopping and optimal liquidation problems, and a model of illiquidity.
 Our discreteness assumption is also natural in some problems in mathematical finance. 

We work in the robust setting introduced by Bouchard and Nutz~\cite{BouchardNutz.13}. This one is designed in order to be able to apply dynamic programming techniques to it; see~\cite{BertsekasShreve.78}. We will follow the formulation of dynamic programming in~\cite{Evstigneev.76}.

When considering the robust dynamic programming procedure, as outlined in the problem of robust utility maximization 
in~\cite{Nutz.13util},
measurability turns out to be the main obstacle to establishing positive results. A robust dynamic programming step consists in broad terms of (1) taking conditional expectations under various measures, i.e. taking infima of those, and (2) maximization. Measurability we are considering is lower-semianalyticity of the map $\Psi$. To show that the first step yields a lower-semianalytic function requires one to be able to approximate the function by a countable number of its function values. This does not even work for concave functions unless one assumes that the domain has a nonempty interior,
%
which is an assumption in~\cite{Nutz.13util}, 
but also explicitly assumed in~\cite{NeufeldSikic.16}. 
So, in the concave case, considered in~\cite{NeufeldSikic.16}, 
step~(1) of taking conditional expectations works, since the operation of taking a conditional expectation preserves concavity of functions, as does taking infima; step~(2) works since concavity is preserved also when maximizing the function in one parameter. The maximization step works even in the setup considered in this paper, given that we assume an appropriate no-arbitrage condition, however step~(1) fails, a minimum of conditional expectations does not need to be measurable. Whereas in the concave case one could infer measurability from concavity, here the conditional expectations have no properties one could exploit.

The no-arbitrage condition we assume is, stronger than the robust no-arbitrage condition, when considering the frictionless market model. Instead of considering concrete models of financial markets and the corresponding no-arbitrage conditions, we opted for quite an abstract approach and, admittedly, strong no-arbitrage condition. We, however, motivate our choices by discussing various examples and by providing intuition for the objects at hand.

Dynamic programming is an approach that replaces the multi-step decision problem with a series of one-step decision problems. If one can solve, i.e.\ prove existence of optimizers of the one-step problems, then one gets existence in the original problem by using those one-step optimizers. One approach in obtaining existence of the one-step problems, taken in the seminal paper \cite{RasonyiStettner.06} as well as in~\cite{Nutz.13util}, is to set up the problem in such a way that the set of strategies in the one-step problems one obtains is compact. Indeed, under a suitable no-arbitrage condition, having a frictionless market model with utility function defined on the positive half-line implies this compactness, up to the projection on the predictable range. We opt for the same approach. This compactness requirement, known as local-level boundedness locally uniformly in~\cite{RockafellarWets}, is imposed by assuming a recession condition on the function $\Psi$. We also refer to \cite{PennanenPerkkio.12}, where in a single-prior convex minimization setting a milder condition on the directions of recession is provided.

Robust utility maximization was already considered in the literature. The closest to our work in discrete-time is~\cite{NeufeldSikic.16}. There, concavity is assumed on the maps $\Psi(\omega,\cdot)$ for every $\omega\in\Omega$. For more results concerning the robust utility maximization problem in a nondominated framework, we refer to \cite{Nutz.13util,Bartl.16,BiaginiPinar.15,CarassusBlanchard.16,DenisKervarec.13,FouquePunWong.16,LinRiedel.14,MatoussiPossamaiZhou.12utility,NeufeldNutz.15,TevzadzeToronjadzeUzunashvili.13}. 
For examples of robust optimal stopping problem,
we refer to~\cite{BayraktarYao.11a,BayraktarYao.11b,NutzZhang.15,EkrenTouziZhang.12stop}. Furthermore, nonconcave utility maximization problems in the classical setup without model uncertainty were considered in~\cite{PennanenPerkkioRasonyi.17,CarassusRasonyi.16,CarassusRasonyiRodrigues.15,Reichlin.13,Reichlin.16}.
To the best of our knowledge, robust utility maximization in the nondominated setting for nonconcave financial markets has not been considered yet. 

The remainder of this paper is organized as follows. In Section~\ref{sec:Multi-Per}, we introduce the concepts, list the assumptions imposed on $\Psi$ and state the main results. Examples are provided in Section~\ref{sec:examples}. In Section~\ref{sec:1-Per}, we introduce and solve the corresponding one-period maximization problem. In Section~\ref{sec:proof-multi-period}, we introduce the notion needed in our dynamic programming approach and explain why this leads to the existence of a maximizer in our optimization problem \eqref{eq:robust-optim-Intro}. The proof is then divided into several steps, which heavily uses the theory of lower semianalytic functions.

\paragraph{\textbf{Notation.}}

For any vector $x\in\R^{dT}$, written out as $x=(x_0,\ldots,x_{T-1})$, where $x_i\in\R^d$ for each $i$, we denote the restriction to the first $t$ entries by $x^t:=(x_0,...,x_{t-1})$. For $y\in \R^{dT}$, we denote by $x\cdot y$ the usual scalar product on $\R^{dT}$. 


\section{Optimization Problem}\label{sec:Multi-Per}
Let $T\in \N$ denote the fixed finite time horizon and let $\Omega_1$ be a Polish space. Denote by $\Omega^t:=\Omega^t_1$ the $t$-fold Cartesian product for $t=0,1,\dots,T$, where we use the convention that $\Omega^0$ is a singleton. Let $\F=(\cF_t)_{t=0,1,\dots, T}$ where $\cF_t:=\bigcap_P \cB(\Omega^t)^P$ is the universal completion of the Borel $\sigma$-field $\cB(\Omega^t)$; here $\cB(\Omega^t)^P$ denotes the $P$-completion of $\cB(\Omega^t)$ and $P$ ranges over the set $\fM_1(\Omega^t)$ of all probability measures on $(\Omega^t,\cB(\Omega^t))$. Moreover, define $(\Omega,\cF):=(\Omega^T,\cF_T)$. This plays the role of our initial measurable space.

For every $t \in  \{0,1,\dots,T-1\}$ and $\omega^t \in \Omega^t$ we fix a nonempty 
set $\fP_t(\omega^t)\subseteq\fM_1(\Omega_1)$ of probability measures; $\fP_t(\omega^t)$ represents the possible laws
for the $t+1$-th period given state $\omega^t$.  Endowing $\fM_1(\Omega_1)$ with the usual topology induced by the weak convergence makes it into a Polish space; see  \cite[Chapter~7]{BertsekasShreve.78}. We assume that for each $t$
\begin{equation*}
\mbox{graph}(\fP_t):=\{(\omega^t,P)\,| \, \omega^t \in \Omega^t,\, P\in \fP_t(\omega^t)\}\ \ \mbox{ is an analytic subset of }\ \ \Omega^t \times \fM_1(\Omega_1) .
\end{equation*}
Recall that a subset of a Polish space is called analytic if it is the image of a Borel subset of a (possibly different) Polish space under a Borel-measurable mapping (see \cite[Chapter~7]{BertsekasShreve.78}); in particular, the above assumption is satisfied if  $\mbox{graph}(\fP_t)$ is Borel. The set $\mbox{graph}(\fP_t)$ being analytic provides the existence of an universally measurable kernel $P_t:\Omega^t\to \fM_1(\Omega_1)$ such that $P_t(\omega^t)\in \fP_t(\omega^t)$ for all $\omega^t \in \Omega^t$ by the Jankov-von Neumann theorem, see \cite[Proposition~7.49, p.182]{BertsekasShreve.78}. Given such a kernel $P_t$ for each $t\in \{0,1,\dots,T-1\}$, we can define a probability measure $P$ on $\Omega$ by 
\begin{equation*}
P(A):=\int_{\Omega_1}\dots \int_{\Omega_1} \mathbf{1}_A(\omega_1,\dots,\omega_T)\, P_{T-1}(\omega_1,\dots,\omega_{T-1};d\omega_T)\dots P_0(d\omega_1), \quad A \in \cF,
\end{equation*}
where we write $\omega:=\omega^T:=(\omega_1,\dots,\omega_T)$ for any element in $\Omega$. We denote a probability measure defined as above by  $P=P_0 \otimes \dots \otimes P_{T-1}$.  For the  multi-period market, we consider the set 
\begin{equation*}
\fP:=\{P_0\otimes\dots\otimes P_{T-1}\,|\,  P_t(\cdot) \in \fP_t(\cdot), \, t=0,\dots,T-1\}\subseteq \fM_1(\Omega),
\end{equation*}
of probability measures representing the uncertainty of the law,
where in the above definition each $P_t:\Omega^t\to \fM_1(\Omega_1)$ is universally measurable such that $P_t(\omega^t)\in \fP_t(\omega^t)$ for all $\omega^t \in \Omega^t$.

We will often interpret $(\Omega^t,\cF_t)$ as a subspace of $(\Omega,\cF)$ in the following way. 
Any set $A \subset \Omega^t$ can be extended to a subset of $\Omega^T$ by adding $(T-t)$ products of $\Omega_1$, i.e. $A^T:=A \times \Omega_1\times\dots \times\Omega_1\subset \Omega^T$. Then, for every measure $P=P_0\otimes\dots\otimes P_{T-1} \in \fP$,
one can associate a measure $P^t$ on $(\Omega^t,\cF^t)$ such that $P^t[A]=P[A^T]$ 
by setting $P^t:=P_0\otimes \dots \otimes P_{t-1}$. 

We call a set $A\subseteq \Omega\,$ $\fP$-polar if for all $P \in \fP$ there exists $A^P \in \cF$ such that $A\subseteq A^P$ and $P[A^P]=0$, 
and say a property to hold $\fP$-quasi surely, or simply $\fP$-q.s., if the property holds outside a $\fP$-polar set;
we will use $\fP$-q.a.\ $\omega$ (quasi all) and $\fP$-q.e.\ $\omega$ (quasi every) to say $\fP$-q.s.. 

A map $\Psi\colon\Omega\times\R^{dT}\rightarrow\overline{\R}$ is called an $\cF$-\emph{measurable normal integrand} if the correspondence $\hypo \Psi:\Omega\rightrightarrows\R^{dT}\times\R$ defined by
$$
  \hypo \Psi(\omega)=\big\{(x,y)\in\R^{dT}\times\R\,\big|\,\Psi(\omega,x)\geq y\big\}
$$ 
is closed-valued and $\cF$-measurable in the sense of set-valued maps, see \cite[Definition~14.1 and Definition~14.27]{RockafellarWets}.
\begin{remark}\label{rem:normal-differ}
	We point out that our definition of a normal integrand $\Psi$ varies from the classical one in  optimization as defined in, e.g., \cite[Chapter 14]{RockafellarWets} in the sense that our map $-\Psi$ satisfies the classical definition of a  normal intergrand. As we are looking for a maximum of a function instead of a minimum like in classical optimization problems, our definition of a normal function fits into our setting.
\end{remark}
 Note that the correspondence $\hypo \Psi$ has closed values if and only if the function $x\mapsto \Psi(\omega,x)$ is upper-semicontinuous for each $\omega$; see \cite[Theorem~1.6]{RockafellarWets}. By \cite[Corollary~14.34]{RockafellarWets}, 
 $\Psi$ is (jointly) measurable with respect to $\cF\otimes\cB(\R^{dT})$ and $\cB(\overline \R)$. Classical examples of normal integrands, which are most prevalent in mathematical finance, are Caratheodory maps; see  \cite[Example~14.29]{RockafellarWets}.

Denote by $\cH$ the set of all $\F$-adapted $\R^d$-valued processes $H:=(H_0,\dots,H_{T-1})$ with discrete-time index $t=0,\dots,T-1$. Our goal is to study the following optimization problem
\begin{equation}\label{eq:optim-prob-multiP}
\sup_{H \in \cH}\inf_{P \in \fP} E^P[\Psi(H_0,\dots,H_{T-1})],
\end{equation}
where $\Psi\colon\Omega\times \R^{d T} \to \overline{\R}$ is an $\cF$-measurable normal integrand.

Recall that a function $f$ from a Borel subset of a Polish space into $[-\infty,\infty]$ is called lower semianalytic if the set $\{f<c\}$ is analytic for all $c\in \R$; in particular any Borel function is lower semianalytic.  Moreover, recall that any analytic set is an element of the universal $\sigma$-field, see e.g. \cite[p.171]{BertsekasShreve.78}. A  function $f\colon\R^n\to \R \cup\{-\infty\}$ is called proper if $f(x)>-\infty$ for some $x \in \R^n$. The domain $\dom f$ of a function $f\colon\R^n\to \R \cup\{-\infty\}$ is defined by 
\begin{equation*}
\dom f\coloneqq\{x \in \R^n\, | \, f(x)>-\infty\}.
\end{equation*}
We refer to \cite{BertsekasShreve.78} and \cite{RockafellarWets} for more details about the different concepts of measurability, selection theorems and optimization theory.

We say that a set $\mathcal{D}\subseteq \R^{dT}$ 
satisfies the \textit{grid-condition} if 
\begin{equation}\label{eq:grid-condition}
\inf_{i\in \{0,\dots,T-1\}} \inf_{x,y\in \cD, x_i\neq y_i} |x_i-y_i|>0.
\end{equation}
The following conditions are in force throughout the paper.
\begin{assumption}\label{ass:Psi-MultiP}
The map $\Psi:\Omega\times \R^{d T} \to \R\cup\{-\infty\}$ satisfies the following:
\begin{enumerate}
\item[(1)] There exists $\mathcal{D}\subseteq \R^{dT}$  satisfying the grid-condition such that
\begin{equation*} 
\mbox{for all $\omega \in \Omega$, $\Psi(\omega,x)=-\infty$ for all $x\notin \cD$;} 
\end{equation*}
\item[(2)] there exists a constant $C\in\R$ such that $\Psi(\omega,x)\leq C$ for all $\omega\in \Omega$, $x \in \R^{dT}$;
\item[(3)] the map $(\omega,x) \mapsto \Psi(\omega,x)$ is lower semianalytic;
\item[(4)]  The zero process $0\in \cH$ satisfies 
$\inf_{P\in\fP} E^P[\Psi(0)]>-\infty$.
%
\end{enumerate}
\end{assumption}
\begin{remark}\label{rem:Psi-usc}
Due to Assumption~\ref{ass:Psi-MultiP}(1), by the grid-condition \eqref{eq:grid-condition} imposed on the set $\cD$,  the map $x \mapsto \Psi(\omega,x)$ is upper-semicontinuous for every $\omega \in \Omega$. In  fact, we prove in the key Lemma~\ref{le:normal} that $\Psi$ is a $\cF$-normal integrand.
\end{remark}
\begin{remark}\label{rem:0-in-D}
Assumption~\ref{ass:Psi-MultiP}(4) ensures that $0\in \mathcal{D}$.
\end{remark}
\begin{remark}\label{rem:bounded above}
At first glance, Assumption~\ref{ass:Psi-MultiP}(2) may seem to be rather restrictive. 
 It was shown in \cite[Example~2.3]{Nutz.13util} that for any (nondecreasing, strictly concave) utility function $U$ being \textit{unbounded from above}, one can construct a frictionless market $S$ and a set $\fP$ of probability measures such that
\begin{equation*}
u(x):=\sup_{H \in \cH}\inf_{P \in \fP} E^P[U(x+ H \bullet S_T)]<\infty
\end{equation*}
for any initial capital $x>0$,
but there is no maximizer $\widehat{H}^x$;  here we denoted by $H\bullet S_T$ the stochastic integral, i.e.\ $H\bullet S_T=\sum_{t=0}^{T-1} \langle H_{t},S_{t+1}-S_t\rangle$. So already in the special case 
$\Psi(H):=U(x+ H \bullet S_T)$,
 the existence may fail for utility functions not being bounded from above. 
 However, Assumption~\ref{ass:Psi-MultiP}(2) is sufficient in order to establish the existence of a maximizer for our optimization problem; see \cite{CarassusBlanchard.16}.
 \end{remark}
%
%
%
%

We now define the horizon function
$\Psi^\infty:\Omega\times\R^{dT}\rightarrow\R\cup\{-\infty\}$ of any function $\Psi:\Omega\times\R^{dT}\rightarrow\R\cup\{-\infty\}$ by
\begin{equation*}
  \Psi^\infty(\omega,x) 
  = \lim_{n\rightarrow\infty}\sup_{\substack{\delta>n,\\|x-y|<\frac1n}}\,\frac1\delta\Psi(\omega,\delta y)
  \end{equation*}
The mapping $\Psi^\infty(\omega,\cdot)$ is positively homogeneous and upper-semicontinuous, see~\cite[Theorem~3.21]{RockafellarWets}. If we assume $\Psi$ to be normal, then so is $\Psi^\infty$, see \cite[Exercise 14.54(a)]{RockafellarWets}.

Throughout the paper we impose the following  condition
\begin{equation}\label{eq:no-arbitrage-multiP}
  \cK:=\{H \in \cH \, | \, \Psi^\infty(H_0,\dots,H_{T-1})\geq 0 \ \fP\mbox{-q.s.}\}=\{0\}
\tag{\ensuremath{\NA(\fP)}}
\end{equation}
which means that for any $H \in \cH$ 
\begin{equation*}
H \in \cK \quad \Longleftrightarrow  \quad H = 0 \ \ \fP\mbox{-q.s.}
\end{equation*}
We call it the no-arbitrage condition. 
The main theorem of this paper is the following.
\begin{theorem}\label{thm:Maxim-exist-MultiP}
	Let $\Psi$ be a map satisfying Assumption~\ref{ass:Psi-MultiP}. If the no-arbitrage condition $NA(\fP)$ holds, then there exists a process $\widehat{H} \in \cH$ such that
	\begin{equation}\label{eq:thm-optimal-MultiP}
	\inf_{P \in \fP} E^P[\Psi(\widehat{H}_0,\dots,\widehat{H}_{T-1})]= \sup_{H \in \cH}\inf_{P \in \fP} E^P[\Psi(H_0,\dots,H_{T-1})].
	\end{equation}
\end{theorem}
We will give the proof of this theorem in Section~\ref{sec:proof-multi-period}.
\begin{remark}\label{rem:why-name-NA}
Let us argue why we call the condition $\cK=\{0\}$ the no-arbitrage condition. Consider first the frictionless market model given by $\Psi(H)=H\bullet S_T$. In this case, the condition $\cK=\{0\}$ is written as
$$
H\bullet S_T\geq0\ \ \fP\mbox{-q.s.} \quad\Longrightarrow\quad \,\,H=0\ \ \fP\mbox{-q.s.}
$$
This condition is strictly stronger than the robust no-arbitrage condition, which one states as: $H\bullet S_T\geq0$ $\fP$-q.s. implies $H\bullet S_T=0$ $\fP$-q.s. In the case where $\fP=\{P\}$, i.e. in the dominated setup, the requirement that $H\bullet S_T=0$ implies $H=0$ is referred to as `there are no redundant assets in the market'. Thus, our no-arbitrage condition in the frictionless market requires the classical no-arbitrage condition and, additionally, that there are no redundant assets. Also, in the general nondominated setup, the condition $\cK=\{0\}$ is in general strictly stronger than the robust no-arbitrage condition when restricted to the frictionless market model.

In market models with friction, the no-arbitrage condition of the form
$$
\Psi(H)\geq0\ \ \fP\mbox{-q.s.} \quad\Longrightarrow\quad \ \ \Psi(H)=0\ \ \fP\mbox{-q.s.}
$$
proved to be insufficient for establishing the fundamental theorem of asset pricing and duality, even when talking about convex transaction costs. Thus, even in the market model with proportional transaction costs one talks about weak no-arbitrage condition, strict no-arbitrage condition and robust no-arbitrage condition; see~\cite{Kabanov2003}. It is only the last concept that proved to be strong enough to imply the fundamental theorem of asset pricing. Note that robust no-arbitrage condition is the (weak) no-arbitrage condition with an additional condition.

Moving further in the direction of more general transaction costs one realizes that it is natural for markets to allow for arbitrage opportunities; this is, e.g., one of the conclusions of the paper~\cite{mete2013utility}.
If that is the case, one says that the market model satisfies the no-arbitrage condition if there are no arbitrage opportunities with `large positions in the market'. Stated differently, it is requested that for every strategy $H$ which is an arbitrage strategy, i.e. $\Psi(H)\geq0$, there exists an $n_0\geq0$, such that $n H$ is not an arbitrage strategy whenever $n\geq n_0$. This is where the horizon function comes into the story. 

The horizon function is, by definition, positively homogeneous. It is, thus, enough to understand no-arbitrage theory for such maps. The reason why the horizon function is useful is that it provides a convenient upper bound for the map $\Psi$. To understand what is meant by that statement, consider, first, a trivial example of a function $f\colon\R\rightarrow\R$ which is upper-semicontinuous and satisfies $f^\infty(x)<0$ whenever $x\not=0$ and $f(0)=f^\infty(0)=0$; this is a market model satisfying our no-arbitrage condition. Since the function $f^\infty$ is positively homogeneous, there exists a continuous positively homogeneous function $g\colon\R\rightarrow\R$ such that $f^\infty(x)<g(x)<0$ for all $x\not=0$. Furthermore, there exists an $a\in\R$, such that $f(x)\leq a+g(x)$ for all $x\in\R$; indeed, we note first that, by Proposition~3.23 in~\cite{RockafellarWets},
$$
  \{x|(f-g)(x)\geq\alpha\}^\infty
  \subseteq
  \{x|(f-g)^\infty(x)\geq0\}
  \subseteq
  \{0\};
$$
the second follows from the assumption that $(f-g)^\infty(x)<0$ on $x\not=0$. We denoted by $A^\infty$, where $A\subset\R^n$, the set $\{x|(\chi_{A})^\infty(x)=0\}$, where $\chi_A(x) = 0$ if $x\in A$ and $\infty$ otherwise. Since the set on the left hand side is non-empty for an appropriately chosen $\alpha$, all inclusions are, in fact, equalities. By Theorem~3.5 in~\cite{RockafellarWets}, we get that $\{x|(f-g)(x)\geq\alpha\}$ is a bounded, hence compact, set. The function $(f-g)$ is upper-semicontinous on a compact set, hence bounded above; outside the compact set it is bounded by assumption. In general, the connection between level boundedness and horizon functions is explained in Theorem~3.31 in~\cite{RockafellarWets}. Although the statements in this simple example are obvious, the thinking extends to the situation considered in this paper: it is the key technical step of the paper to show that the no-arbitrage condition we consider, $\cK=\{0\}$ translates to the local one, defined in the following pages.

To conclude, the no-arbitrage condition we are proposing, $\cK=\{0\}$, is implied by the following, more natural, condition: There exists a map $\Upsilon\colon\Omega\times\R^{dT}\rightarrow\overline\R$, such that $\Upsilon(\omega,\cdot)$ is positively homogeneous, continuous, $\Upsilon(\omega,x)>\Psi^\infty(\omega,x)$ $\fP$-q.s.\ and the map $\Upsilon$ satisfies the following condition
$$
\Upsilon(H)\geq0\ \ \fP\mbox{-q.s.}
\quad\Longrightarrow\quad 
\Upsilon(H)=0\ \ \fP\mbox{-q.s.},
$$
i.e. the map $\Upsilon$ satisfies a weak form of no-arbitrage condition. One says that the map $\Upsilon$ dominates $\Psi$. This condition is reminiscent of the robust no-arbitrage condition in markets with proportional transaction costs with efficient friction. The condition $\cK=\{0\}$ is also implied by the following: there exists a frictionless market model $(S_t)$ satisfying our no-arbitrage condition, i.e. which is arbitrage free and there are no redundant assets, and a random variable $a$ such that $\Psi(H)\leq a+H\bullet S_T$ for all $\omega$, $x$. Thus, the reason for choosing such an abstract condition to serve as a no-arbitrage condition is merely because it puts various situations under the same roof.	
\end{remark}
\begin{remark}
  It remains to mention why the no-arbitrage is a $\fP$-q.s.\ condition. The idea behind is simple: Consider the frictionless market model. If we do not know the probability distributions of the increments of the stock price process precisely, we consider the worst case. For a given strategy, the worst case would be the infimum of the value at final time over all the measures. However, due to measurability issues, this is not a well-defined object. So, the market model could admit arbitrage for every element $P\in\fP$, but still be arbitrage free; see Example~\ref{ex:NA stuff}.
\end{remark}

%
%
\section{Examples}\label{sec:examples}

\subsection{Integer Valued Strategies and Semi-Static Trading}
Consider $S=(S^1,\ldots,S^d)$ a $d$-dimensional stock price process whose components are Borel measurable for each $t$. Consider $U\colon\Omega\times\R\rightarrow\R\cup\{-\infty\}$ a random (not necessarily concave) utility function; precisely, $U(\omega,\cdot)$ is a continuous function which is bounded above by some constant $C$ for each $\omega\in\Omega$. Furthermore, assume that $(\omega,x)\mapsto U(\omega,x)$ is lower-semianalytic. We also assume that the trader is only allowed to have integer positions in the risky assed, i.e. $h_t(\omega)\in\Z^d$ for all $\omega$, $t$. We define a mapping
$$
\Psi(h)\coloneqq U\bigg(x + \sum_{t=0}^{T-1}\langle h_t, S_{t+1}-S_t\rangle\bigg) - \chi_{\Z^{dT}}(h), 
$$
where $x\in\R$ is the fixed initial wealth of the trader and $\chi_A(x)$ is the  function taking value $0$ if $x\in A$ and $\infty$ otherwise. Let us show that $\Psi$ satisfies Assumption~\ref{ass:Psi-MultiP}. 
%
First, the grid-condition is clear.
Also lower-semianalyticity is clear, as it is a composition of a lower-semianalytic function and a Borel one. Boundendess from above follows by the same assumption on $U$. 
Finally, if $\inf_{P\in\fP}E^P[U(x)]>-\infty$, then also the last condition holds and hence $\Psi$ indeed satisfies Assumption~\ref{ass:Psi-MultiP}.
 

\medskip
We may also consider the utility maximization with semi-static portfolios where beside a position $h$ in stock one can also hold a position in static assets $\{f_i\,|\,i=1,\ldots,I\}$ available for free in the market. The value of a strategy $(h,g)$ is given by
$$
V(h,g) = x + \sum_{t=0}^{T-1}\langle h_t, S_{t+1}-S_t\rangle + \sum_{i=1}^{I}g_if_i.
$$
In a similar way as above one shows that 
$$
\Psi(h,g)\coloneqq U\bigg(x + \sum_{t=0}^{T-1}\langle h_t, S_{t+1}-S_t\rangle + \sum_{i=1}^{I}g_if_i\bigg) - \chi_{\Z^{dT}}(h) - \chi_{\Z^I}(g)
$$
satisfies Assumption~\ref{ass:Psi-MultiP}.

\subsection{The Optimal Stopping Problem}
Given a process $G\coloneqq(G_t)_{t=0,\ldots,T}$, with $G_t$ an $\cF_t$-measurable random variable for each $t$, we consider an optimal stopping problem. Denote by $\cT$ the set of $\F$-stopping times; the problem is then
$$
\inf_{P\in\fP} E^P[G_\tau]\longrightarrow\max\qquad\text{over }\tau\in\cT.
$$
To transform it into a standard form, note that $\tau$ is a stopping time if and only if $\{\tau=t\}\in\cF_t$ for all $t$, i.e. if the process $\1_{[\!\![ 0,\tau)\!\!)}$ is adapted. Equivalently, every adapted $\{0,1\}$-valued decreasing process $h$ defines a stopping time $\tau\coloneqq\inf\{t\,|\,h_t=0\}.$ We define the following mapping
$$
\Psi(h)\coloneqq\sum_{t=0}^T(h_{t-1}-h_t)\,G_t - \chi_D(h),
$$
where $D\subset \Z^{T+2}$ is defined as follows
$$
D = \big\{x=(x_{-1},\ldots,x_T)\,\big|\,x_{t-1}\geq x_t,\ x_t\in\{0,1\}\ \ \forall t, \ x_{-1}=1, \ x_{T}=0\big\}.
$$
The set $D$ is, clearly, closed; consequently, the map $\chi_D$ is lower semicontinuous, hence Borel. Hence the map $(\omega,x)\mapsto\Psi(\omega,x)$ is lower semianalytic as soon as $G_t$ are.
 The domain of the function $\Psi(\omega)$ is finite; this implies also that $\Psi$ satisfies the grid-condition. There exists an upper bound $C$ as soon as $G_t\leq C$ $\fP$-q.s.. Finally, if  $\inf_{P\in\fP}E^P[G_0]>-\infty$, then also the last condition holds.

One could also model optimal liquidation problems, where a trader starts with a large integer position $h_{-1}=M \in \N$ and needs to liquidate it by the final time, i.e. $h_T=0$; see e.g. \cite{AlmgrenChriss.00}.

%

%

\subsection{Roch--Soner Model of Illiquidity}
Here we give a more involved example of a limit order book. More information about modeling considerations can be found in~\cite{roch2013resilient}; see also \cite{mete2013utility} for a discrete-time version of the model. The `equilibrium stock price' process $S\coloneqq(S_t)$ represents the price when there is no trading; with trading, after the trading period, the price is given by $S_t+\ell_t$; $\ell_t$ represents impact of trading. If a big trader executes the trade $\Delta h_t$ at time $t$, he or she moves the price, so the price after trade changes by $m_t \Delta h_t$. The model is the following
\begin{align*}
\ell_{t+1} &= (1-\kappa)\ell_t + 2m_{t+1}\Delta h_{t+1}\\[2pt]
V_{t+1}    &= V_t + h_t\Delta S_{t+1} - \kappa\ell_th_t-\Delta m_{t+1}h_t^2 \\[2pt]
V_0 &= \ell_0 = 0
\end{align*}
for some constant $\kappa\in(0,1)$, capturing the decay of price impact $\ell$, and a strictly positive process $(m_t)$, encoding the `depth' of the limit order book of $1/2m_t$. 

Let $U\colon\Omega\times\R\rightarrow\R\cup\{-\infty\}$ be a random utility function. Set
$$
\widehat\Psi(h)\coloneqq U(V_T).
$$
One can show that the mapping $h\mapsto V_T$ is not concave. Next, we restrict trading to  integer values, i.e. define $\Psi(h) = \widehat\Psi(h)-\chi_{\Z^T}(h)$. Let us now check the conditions. 
It is easy to see that the grid condition is satisfied for the map $\Psi$. Similarly for the existence of the upper bound: it holds as soon as it holds for $U(\omega,\cdot)$ for all $\omega\in\Omega$. The map $(\omega,x)\mapsto\Psi(\omega,x)$ will be lower semianalytic if $S_t$ and $m_t$ are Borel measurable random variables and the map $(\omega,x)\mapsto U(\omega,x)$ is lower semianalytic. The last condition holds as soon as $\inf_{P\in\fP}E^P[U(0)]>-\infty$.

%
%
%
\subsection{An Example of a Set $\fP$}
\begin{example}
\label{ex:NA stuff}
Consider a one-step frictionless market model on the measurable space $\Omega=[0,\frac13]\cup[\frac23,1]$ with the sigma-algebra $\cF_0=\{\varnothing,\Omega\}$. Consider a stock price $S$, given by $S_0=1$ identically, and $S_1(\omega)=2\omega$. Consider the family 
$$
  \fP_0 = \{\delta_x\,|\,x\in\Omega\},
$$
where by $\delta_x$ we denoted a Dirac measure with all the mass in $x$. The set $\Omega$ is, clearly, a metric space as a subset of $\R$ and the random variable $S_1$ is Borel-measurable. Let $h\in\R$ be a strategy; assume that $h\not=0$. If $h>0$ and $P=\delta_x$, with $x>\frac12$, the strategy $h$ is an arbitrage strategy
$$
  P[h(S_1-S_0)] = h(2x-1)>0.
$$
However, if $x<\frac12$, $h$ is not an arbitrage strategy. One can repeat the analysis for $h<0$. What this shows is that the market model may admit arbitrage for every $P\in\fP_0$, but still be arbitrage free. One can see that this, classical definition of no arbitrage is equivalent to the one we are considering; see Remark~\ref{rem:why-name-NA}.

Let us show that the set of probability measures $\fP_0$ has analytic graph
$$
  \mbox{graph}(\fP_0)=\{(\omega^0,P)\,|\,\omega^0\in\Omega^0,\ P\in\fP_0(\omega^0)\}.
$$
Given that $\Omega^0$ is a singleton and the set $\fP_0$ clearly closed, it is also analytic.
\end{example}
A similar analysis as in the previous example could be repeated for other examples. Let us, thus, only concentrate on the sets of measures in $\fP$.

In fact, we only need to focus on the single step correspondences $\omega^t\mapsto\fP_t(\omega^t)$ with analytic graph. An example of those would be correspondences with Borel, closed or open graph, like upper and lower hemicontinuous correspondences.

Another way of inducing the set of probability measures $\fP_t$ is to `guess' the true transition function $\omega^t\mapsto P_t(\omega^t)$ and define $\fP_t(\omega^t)$ to be its neighbourhood, using, e.g. a continuous function $F\colon\Omega^t\times\cP(\Omega_1)\times\cP(\Omega_1)\rightarrow[0,\infty)$ as follows
\[
  \fP_t(\omega^t)\coloneqq\big\{P\in\cP(\Omega_1)\big|F\big(\omega^t,P,P_t(\omega^t)\big)\leq\varepsilon\big\}
\]
for some $\varepsilon>0$.
Examples of the function $F$ include the various distances between measure spaces, e.g. $F(\omega^t,P,Q) =\|P-Q\|_{TV}$.
\begin{example}
Let us sketch another procedure of specifying the set of measures $\fP$ on an example: The $\fP$ determines a parametric family of probability measures with unknown parameter. Consider a frictionless market model in one step with $\Omega=[0,1]$. Assume that the stock price is given by $S_0=1$ and $S_1(\omega)=2\omega$. We want to specify the set $\fP$ such that under each measure $P\in\fP$ the stock price is a binomial tree model where the initial stock price process $S_0=1$; the stock price goes up with probability $p\in [0.3,0.7]$ and takes value in $[1.4,1.6]$, or the stock price goes down and takes value in $[0.4,0.6]$. This can be modeled as follows: Set $\Omega\coloneqq[0.4,0.6]\cup[1.4,1.6]$, $S_1(\omega)=\omega$ and define 
$$
\fP\coloneqq\big\{p\delta_u + (1-p)\delta_d\,\big|\,d\in[0.4,0.6],\ u\in[1.4,1.6],\ p\in[0.3,0.7]\big\};
$$
where by $\delta_x$ we denoted a Dirac measure with all the mass in $x$. Under every measure $P\in\fP$ the market model $S$ is arbitrage free under $P$; indeed, the no-arbitrage condition in Remark~\ref{rem:why-name-NA} is trivially true in this example.

To conclude, we need to show that $\mbox{graph}(\fP_0)$ is analytic. But it is clear that it is sequentially closed; since $\fP$ is a subset of a compact metric space, it is also closed, hence also analytic. 
\end{example}

Section~2.3 in~\cite{Bartl.16} contains a more elaborate discussion on construction of the set $\fP$.
%

\section{One-Period-Model}\label{sec:1-Per}
\subsection{Setup}\label{subsec:1-Per-setup}
Let $(\Omega,\cF)$ be a measurable space
and $\fP$ be a  possibly nondominated set of probability measures on it. Fix a function $\Psi \colon \Omega\times \R^{d} \to \R\cup\{-\infty\}$ and consider the optimization problem
\begin{align}\label{eq:optimization-problem-1-Per}
\sup_{h\in\cH} \inf_{P\in\fP} E^P\big[\Psi(h)\big];
\end{align}
the set of strategies is here just $\cH=\R^d$.


\begin{assumption}\label{ass:1-Per}
The map $\Psi\colon\Omega\times\R^d\rightarrow\R\cup\{-\infty\}$ satisfies 
\begin{enumerate}
\item[(1)] The map $x\mapsto\Psi(\omega,x)$ is upper-semicontinuous for all $\omega \in \Omega$ and\\ \, the map 
$(\omega,x)\mapsto\Psi(\omega,x)$ is $\cF\otimes \cB(\R^d)$-measurable.

%
%
%
\item[(2)] there exists a constant $C\in\R$ such that $\Psi(\omega,x)\leq C$ for all $\omega\in \Omega$, $x \in \R^d$;
\item[(3)]  We have that $\inf_{P\in\fP} E^P[\Psi(0)]>-\infty$.

\end{enumerate}
\end{assumption}
\begin{remark}\label{rem:meas-1-per}
Note that Assumption~\ref{ass:1-Per}(1) is much weaker than Assumption~\ref{ass:Psi-MultiP}(1) which involves a set $\cD\subseteq \R^{dT}$ satisfying the grid-condition; see also 
Lemma~\ref{le:normal}. 
Moreover,  in the one-period model we did not require any structural properties on the measurable space $(\Omega,\cF)$ nor on the set $\fP$ of probability measures. In particular, we did not assume that the map $(\omega,x)\mapsto \Psi(\omega,x)$ is lower semianalytic.
The stronger Assumption~\ref{ass:Psi-MultiP} is necessary in the multi-period model for the purpose of ensuring  lower semianalyticity of maps appearing in a dynamic programming procedure, so that measurable section results can be applied; see (3) in the proof of Proposition~\ref{Psi-DPP-Ass-Multi-P-OK}.
\end{remark}

We work under the following no-arbitrage condition 
\begin{align}\label{eq:def-NA-1-Per}
 \cK := \big\{h\in\cH\,\big|\,\Psi^\infty(h)\geq 0\quad\fP\textrm{-q.s.}\big\}=\{0\}.
\end{align}
\begin{theorem}\label{thm:1-Per}
  Let the no-arbitrage condition \eqref{eq:def-NA-1-Per} and Assumptions~\ref{ass:1-Per} hold. Then there exists a strategy $\widehat{h} \in \R^d$ such that
\begin{equation}\label{eq:thm:1-Per}
 \inf_{P\in\fP} E_P\big[\Psi(\widehat h)\big]= \sup_{h\in\R^d}\inf_{P\in\fP} E_P\big[\Psi(h)\big].
\end{equation}
\end{theorem}
The proof of Theorem~\ref{thm:1-Per} will be provided in the following  Subsection~\ref{subsec:proof-1-Per}.
%
%

%
%
\subsection{Proof of the One-Period Optimization Problem}\label{subsec:proof-1-Per}
The key result from \cite{RockafellarWets} which we use to prove the one-period optimization problem is the following:
\begin{proposition}\label{prop:key-argument-RW}
 Let $f:\R^m\times\R^n\to \R\cup\{-\infty\}$ be proper, upper-semicontinuous with 
 here $\cK:=\{x\in\R^n\,|\,f^\infty(0,x)\geq0\}=\{0\}$. Then the function
 $$
   p(u) = \sup_{x\in\R^n}f(u,x)
 $$ 
 is proper, upper-semicontinuous, and for each $u\in dom(p)$ there exists a maximizer $x(u)$. Moreover, we  have $p^\infty(u) = \sup_{x\in\R^n}f^\infty(u,x)$, which is attained whenever $u\in dom(p^\infty)$.
 \end{proposition}
\begin{proof}
This is simply the summary of \cite[Theorem~3.31, p.93]{RockafellarWets} together with \cite[Theorem~1.17, p.16]{RockafellarWets}.
\end{proof}

Consider the function  
\begin{equation*}
\Phi\colon h\mapsto\inf_{P\in\fP} E^P[\Psi(h)].
\end{equation*}
 Observe that   $\Phi$ is upper-semicontinuous as an infimum of upper-semicontinuous functions. It is also proper, i.e. not identically equal to $-\infty$, by Assumption~\ref{ass:1-Per}(3). Moreover, we have the following.

\begin{lemma}\label{le:psi-phi-infty}
Let $\Psi\colon\Omega\times \R^{d} \to \R\cup\{-\infty\}$ satisfy Assumption~\ref{ass:1-Per}. Then 
\begin{equation*}
\Psi^{\infty}(\omega,h)\leq 0, \quad \quad \mbox{and}  \quad\quad \Phi^\infty(h)\leq \inf_{P \in \fP} E^P[\Psi^\infty(h)]\leq 0 \quad\quad  \forall \omega \in \Omega, h \in \R^d.
\end{equation*}
 Furthermore, if the no-arbitrage condition \eqref{eq:def-NA-1-Per} holds, then 
\begin{equation*}
\Phi^\infty(h) =0 \ \Longleftrightarrow \
\Psi^\infty(h)=0 \ \fP\mbox{-q.s.}  \ \Longleftrightarrow \ h=0.
\end{equation*}
\end{lemma}
\begin{proof}
As taking the limit in $n$ in the expression of $\Phi^\infty$ is the same as taking the $\inf_{n \in \N}$, and by monotone convergence theorem, as $\Psi\leq C$ by assumption, we have
\begin{align*}
  \Phi^\infty(x) 
  &=
  \inf_{n \in \N}\,\, \sup_{\substack{\delta>n,\\|x-y|<\frac1n}}\,\frac1\delta \inf_{P \in \fP}E^P\big[\Psi(\delta y)\big]
  \leq
  \, \inf_{P \in \fP}\inf_{n \in \N}E^P\bigg[\sup_{\delta>n,\ |x-y|<\frac1n}\,\frac1\delta\Psi(\delta y)\bigg]
  \\[1ex]&=
  \,\, \inf_{P \in \fP}E^P\bigg[\lim_{n\rightarrow\infty}\sup_{\delta>n,\ |x-y|<\frac1n}\,\frac1\delta\Psi(\delta y)\bigg]
  =\inf_{P \in \fP} E^P\big[\Psi^\infty(x)\big].
\end{align*}
As $\Psi$ is uniformly bounded from above, we have $\Psi^\infty\leq 0$ which then by the above inequality also implies that $\Phi^\infty\leq 0$. In particular, $\Phi^\infty=0$ implies that $\Psi^\infty=0 \ \fP$-q.s. and hence by definition $h \in  \cK$. The no-arbitrage condition \eqref{eq:def-NA-1-Per} now implies that $h=0$.

For the other direction, let $h=0$. Then as $\Psi\leq C$  and $\inf_{P\in\fP} E^P[\Psi(0)]>-\infty$ by assumption, we have that
\begin{align*}
\Phi^\infty(0)
=\lim_{n\rightarrow\infty}\sup_{\delta>n,\ |y|<\frac1n}\,\frac1\delta\inf_{P \in \fP}E^P\big[\Psi(\delta y)\big]
\geq \lim_{n\rightarrow\infty}\sup_{\delta>n}\,\frac1\delta\inf_{P \in \fP}E^P\big[\Psi(0)\big]
=0.
\end{align*}
\end{proof}
\begin{remark}\label{rem:1-per-le-no-usc}
	Notice that in Lemma~\ref{le:psi-phi-infty}, the assumption that the map $x\mapsto \Psi(\omega,x)$ is upper-semicontinuous for all $\omega \in \Omega$ is not necessary. 
\end{remark}

\begin{proof}[Proof of Theorem~\ref{thm:1-Per}]
By the no-arbitrage condition \eqref{eq:def-NA-1-Per} and Lemma~\ref{le:psi-phi-infty}, we see that the conditions of
Proposition~\ref{prop:key-argument-RW} are fulfilled for $f\equiv\Phi$, which gives us the existence of a maximizer $\widehat h$, as by Assumption~\ref{ass:1-Per}(3)
\begin{align*}
\sup_{h \in \R^d} \Phi(h)
= 
\sup_{h \in \R^d} \inf_{P \in \fP} E^P[\Psi(h)]
\geq\inf_{P \in \fP} E^P[\Psi(0)]
>-\infty.
\end{align*}
\end{proof}


%
%
\section{Multi-Period-Model}\label{sec:proof-multi-period}
We denote by $\omega^t\otimes_t\tilde{\omega}$ the pair $(\omega^t,\tilde{\omega})\in \Omega^{t+1}$, where $\omega^t \in \Omega^t$  and $\tilde{\omega}\in \Omega_1$, and define the following sequences of maps: set $\Psi_{T}:=\Psi$ and for $t=T-1,\dots,0$ and $\omega^t \in \Omega^t$ define
\begin{align}
\Phi_t(\omega^t,x^{t+1})
&:=\inf_{P \in \fP_t(\omega^t)}E^P[\Psi_{t+1}(\omega^t\otimes_t\,\cdot,x^{t+1})], \nonumber\\
\Psi_t(\omega^t,x^{t})
&:=\sup_{\tilde{x}\in \R^d} \Phi_t(\omega^t,x^{t},\tilde{x}).
\label{eq:Multi-P-Recursion}
\end{align}
%
%
%
%
%
%
For each $t \in \{0,\dots,T-1\}$, denote by $\cH^t$ the set of all $\F$-adapted, $\R^d$-valued processes $H^t:=(H_0,\dots,H_{t-1})$; these are just restrictions of strategies in $\cH$ to the first $t$ time steps. In the same way, for any set $\cD\subseteq\R^{dT}$ define the subset $\cD^t\subseteq \R^{dt}$ to be the projection of $\cD$ onto the first $t$ components, i.e. first $dt$ coordinates.
%
%
We start with the following simple but important result.
\begin{proposition}\label{Psi-DPP-Ass-Multi-P-OK}
If $\Psi$ satisfies Assumption~\ref{ass:Psi-MultiP}, then for any $t \in \{0,\dots,T-1\}$ also the functions $\Psi_{t+1}$ and $\Phi_t$ satisfy Assumption~\ref{ass:Psi-MultiP}. 
\end{proposition}

Next, define the no-arbitrage condition up to time $t$, denoted by $\mbox{NA}(\fP)^t$, for the mappings $(\Psi_t)$ in the natural way, by saying
\begin{align*}
  \cK^t := \big\{H^t\in\cH^t\,\big|\,\Psi_t^\infty(H^t)\geq 0\quad\fP\textrm{-q.s.}\big\}=\{0\}.
\end{align*}
Condition $\mbox{NA}(\fP)^t$ is a statement about a set of strategies and as such cannot yet be used to prove things that we need it for. What we need is a local version of the no-arbitrage condition. 
\begin{definition}\label{def:NA-t}
	For each $t \in \{0,\dots,T-1\}$ and $\omega^t \in \Omega^t$ define a set 
	\begin{equation}\label{eq:def-K-t-local}
	K_t(\omega^t)
	:=\{h \in \R^{d} \,|\, \Phi^\infty_{t}(\omega^t,0,\dots,0,h)\geq 0\}.
	\end{equation}
	We say that condition $\mbox{NA}_t$ holds if
\begin{align*}
\big\{\omega^t \in \Omega^t \, \big|\, K_t(\omega^t)=\{0\} \big\} \  \mbox{ has }\fP\mbox{-full measure.}
\end{align*}
\end{definition}

\begin{example}
  To briefly sketch what this condition means, consider the frictionless market model with $t=T-1$, initial capital $x>0$ and $U\colon(0,\infty)\rightarrow\R$. In this case the utility is defined by $\Psi(H)\coloneqq U(x + \sum_t H_t(S_{t+1}-S_t))=\Psi_T(H)$. Given that the utility function is concave and bounded above, we get, assuming Inada conditions, that $\Psi_T^\infty(H) = 0$ whenever $\sum_t H_t(S_{t+1}-S_t)\geq0$ and $-\infty$ otherwise. Thus, by definition, $K_t(\omega^t) = \{h\in\R^d\,|\,h(S_{t+1}(\omega^t,\cdot) - S_t(\omega^t))\geq 0\ \fP_t(\omega^t)-q.s.\}.$ We refer the reader to~\cite{NeufeldSikic.16} for further details and discussion of the condition.
\end{example}

%
%
\begin{proposition}\label{prop:local-NA}
	Let $t\in\{0,\ldots,T-1\}$. Assume that $\Psi_{t+1}$ satisfies Assumption~\ref{ass:Psi-MultiP}. If the no-arbitrage condition  $\mbox{NA}(\fP)^{t+1}$ up to time $t+1$ holds, 
	then the local no-arbitrage condition $\NA_t$ holds.
\end{proposition}
Having this result at hand, one can proceed as in the one-step case. The following is a direct consequence.
\begin{proposition}\label{prop:Psi-MultiP-OneP}
	Let $t \in \{0,\dots,T-1\}$. Assume that 
	$\Psi_{t+1}$ satisfies Assumption~\ref{ass:Psi-MultiP} and that
	$\mbox{NA}(\fP)^{t+1}$ holds. Then for every $H^t \in \cH^t$ such that $\sup_{x\in \R^d}\Phi_t(\omega^t,H^t(\omega^{t-1}),x)>-\infty$ \ $\fP\mbox{-q.s.}$ there exists an 
			$\cF_t$-measurable mapping $\widehat{h}_t:\Omega^t\to\R^d$ satisfying
			\begin{equation*}
			\Phi_t(\omega^t,H^t(\omega^{t-1}),\widehat{h}_t)
			=\Psi_t(\omega^t,H^t(\omega^{t-1})) 
			\qquad
			\mbox{for }\fP\mbox{-q.e. }\omega^t \in \Omega^t.
			\end{equation*}
\end{proposition}
\begin{remark}\label{rem:normal}
Compared to the one-step case where the optimization is simply over $\R^d$, there is the additional technical issue that the optimizer $\widehat{h}_t(\omega^t)$, whose pointwise existence we obtain from  the one-step case in Section~\ref{sec:1-Per}, is measurable in $\omega^t$. A key observation which is helpful for this so-called measurable selection problem is the observation in Lemma~\ref{le:normal} that both $\Psi_{t}$ and $\Phi_t$ are normal integrands 
as a direct consequence of  Assumption~\ref{ass:Psi-MultiP}.
\end{remark}
The  important step toward the proof of our main result is the observation that the  no-arbitrage condition  $\mbox{NA}(\fP)^{t}$ up to time $t$ behave well under the dynamic programming recursion. 
\begin{proposition}\label{prop:no-arbitrage-up-time-t-MultiP}
  Let $t\in \{0,\ldots,T-1\}$ and let $\Psi_{t+1}$ satisfy Assumption~\ref{ass:Psi-MultiP}. If the no-arbitrage condition  $\mbox{NA}(\fP)^{t+1}$ up to time $t+1$ holds, then so does the no-arbitrage condition  $\mbox{NA}(\fP)^{t}$ up to time $t$.
\end{proposition}
The proofs of Propositions~\ref{Psi-DPP-Ass-Multi-P-OK}  \& \ref{prop:local-NA}--\ref{prop:no-arbitrage-up-time-t-MultiP} will be given in the Subsection~\ref{subsec:global-local-prop-proof}.
%
%
%
%


%
\subsection{Proofs of Propositions~\ref{Psi-DPP-Ass-Multi-P-OK} \& \ref{prop:local-NA}--\ref{prop:no-arbitrage-up-time-t-MultiP}}\label{subsec:global-local-prop-proof}
We first start with a  useful lemma providing the relation between $\Psi_{t+1}$ and $\Phi_t$.
\begin{lemma}\label{le:ass-psi-phi-multiP}
Let $t \in \{0,\dots,T-1\}$ and $\Psi_{t+1}$ satisfies Assumption~\ref{ass:Psi-MultiP}. Then 
\begin{itemize} 
\item[1)] $\Phi_{t}$ satisfies Assumption~\ref{ass:Psi-MultiP}(4);

\item[2)] For all $(\omega^t\otimes_t\tilde{\omega},x^{t+1})\in \Omega^{t+1}\times \R^{d(t+1)}$ we have
\begin{equation*}
\Psi_{t+1}^{\infty}(\omega^t\otimes_t\tilde{\omega},x^{t+1})\leq 0, \quad  \mbox{and}  \quad \Phi_t^\infty(\omega^t,x^{t+1})\leq \inf_{P \in \fP_t(\omega^t)} E^P[\Psi_{t+1}^\infty(\omega^t\otimes_t\cdot,x^{t+1})]\leq 0.
\end{equation*}
\end{itemize}
\end{lemma}
\begin{proof}
For the first part, 
%
let $H^{t+1}\in \cH^{t+1}$ such that
\begin{equation*}
\inf_{P \in \fP} E^P\big[\Psi_{t+1}(H^{t+1})]>-\infty.
\end{equation*}
We claim that also
\begin{equation*}
\inf_{P \in \fP} E^P\big[\Phi_{t}(H^{t+1})]>-\infty.
\end{equation*} 
To see that, note that by 
 \cite[Proposition~7.50]{BertsekasShreve.78}, there exists for any $\varepsilon>0$ an universally measurable kernel $P^\varepsilon_t:\Omega^t \to \fM_1(\Omega_1)$ such that $P^\varepsilon_t(\omega^t) \in \fP_t(\omega^t)$ for all $\omega^t \in \Omega^t$ and 
\begin{equation*}
E^{P^\varepsilon_t(\omega^t)}[\Psi_{t+1}(\omega^t\otimes_t\cdot, H^{t+1}(\omega^t))]\leq
 \begin{cases}
    \Phi_{t}(\omega^t,H^{t+1}(\omega^t)) + \eps & {\rm{if }}\ \Phi_{t}(\omega^t,H^{t+1}(\omega^t))>-\infty, \\
    -\eps^{-1} & {\rm otherwise.}
  \end{cases}
\end{equation*}
Then, we have
\begin{align*}
\ E^{P^\varepsilon_t(\omega^t)}[\Psi_{t+1}(\omega^t\otimes_t\,\cdot, H^{t+1}(\omega^t))]-\varepsilon
\leq  \ (-\varepsilon^{-1})\vee \Phi_{t}(\omega^t,H^{t+1}(\omega^t)).
\end{align*}
Take any $P \in \fP$ and denote its restriction to $\Omega^t$ by $P^t$. Integrating the above inequality yields
\begin{align*}
E^{P^t}[(-\varepsilon^{-1})\vee \Phi_{t}(H^{t+1})]
\geq
E^{P^t\otimes P^\varepsilon_t}[\Psi_{t+1}( H^{t+1})]-\varepsilon 
\geq
\inf_{P'\in \fP}E^{P'}[\Psi_{t+1}( H^{t+1})]-\varepsilon.
\end{align*}
Letting $\varepsilon \to 0$, we obtain, by Fatou's Lemma, that
\begin{equation*}
E^P[ \Phi_{t}(H^{t+1})] \geq \inf_{P'\in \fP}E^{P'}[\Psi_{t+1}( H^{t+1})]>-\infty.
\end{equation*}
Hence by the arbitrariness of $P \in \fP$, the claim is proven and so $\Phi_{t}$ satisfies Assumption~\ref{ass:Psi-MultiP}(4).  

The second part follows by  
%
 Lemma~\ref{le:psi-phi-infty}.
\end{proof}
%
%
%
%
We continue with the proof of Proposition~\ref{Psi-DPP-Ass-Multi-P-OK}.
\begin{proof}[Proof of Proposition~\ref{Psi-DPP-Ass-Multi-P-OK}]
	%
	We argue backwards using the recursion defined in \eqref{eq:Multi-P-Recursion}.
	The numbering refers to the corresponding numbering in Assumption~\ref{ass:Psi-MultiP}.
	\begin{enumerate}
		\item[{\bf(1)}] Since $\Psi=\Psi_T$ satisfies Assumption~\ref{ass:Psi-MultiP}(1) with respect to some set $\cD\subseteq \R^{dT}$ satisfying the grid-condition~\eqref{eq:grid-condition}, we get directly from its definition that both $\Phi_t$ and $\Psi_t$ are equal to $-\infty$ outside of $\cD^{t+1}$ and $\cD^t$, respectively. 
		Indeed,  if $\Psi_{t+1}$ equals to $-\infty$ outside of  $\cD^{t+1}$, then also does $\Phi_t$ by its definition. Moreover, if $x^t\not \in \cD^t$ then for all $\tilde{x} \in \R^d$ we have that $(x^t,\tilde{x})\notin \cD^{t+1}$ and hence we obtain that also  $\Psi_t$ is equal to $-\infty$ outside of $\cD^t$.
		As a consequence, as also each $\cD^t$ satisfies the grid-condition~\eqref{eq:grid-condition}, the maps $x \mapsto \Psi_t(\omega,x)$ and $x \mapsto \Phi_t(\omega,x)$ are upper-semicontinuous; see Remark~\ref{rem:Psi-usc}.
		\item[{\bf(2)}] Since $\Psi=\Psi_T$ is bounded from above by a constant $C$, the same holds true for $\Phi_t$ and $\Psi_t$ by definition, with respect to the same constant $C$.
		\item[{\bf(3)}] One proves that $\Phi_t$ is lower semianalytic if $\Psi_{t+1}$ is in exactly the same way as was done in
		\cite[Lemma~5.9]{NeufeldSikic.16}; hence 
		the proof is omitted. Assume now that $\Phi_t$ is lower semianalytic; we will argue that then also $\Psi_t$ is. Denote by $\cD_{t+1}$ 
		the projection of $\cD^{t+1}$ onto its last component, i.e. the projection of $\cD$ onto its $t+1$ component. Then, as $\cD^{t+1}$ is at most countable, the same holds true for $\cD_{t+1}$. As by {\bf(1)}, $\Phi_t$ equals to $-\infty$ outside of $\cD^{t+1}$, we have for any $x^t\in \R^{dt}$ that
		\begin{equation*}
		\Psi_t(\omega,x^t):=\sup_{x\in \R^d} \Phi_t(\omega,x^t,x)=\sup_{x\in \cD_{t+1}} \Phi_t(\omega,x^t,x).
		\end{equation*}
		Therefore, being a countable supremum over lower semianalytic functions,  the map $(\omega,x^t) \mapsto \Psi_t(\omega,x^t)$ is lower semianalytic. Note that it is here where we use the grid-condition in a significant way; see also Remark~\ref{rem:meas-1-per}.
		\item[{\bf(4)}] In Lemma~\ref{le:ass-psi-phi-multiP}, we show that $\Phi_t$ satisfies Assumption~\ref{ass:Psi-MultiP}(4) if $\Psi_{t+1}$ does. Then, by definition, this implies that also $\Psi_{t}$ satisfies Assumption~\ref{ass:Psi-MultiP}(4).
	\end{enumerate}	
\end{proof}
%
The following lemma proves that  for each  $t \in \{0,\dots,T-1\}$, both $\Psi_{t+1}$ and $\Phi_t$ are normal integrands. This is the key property needed in the measurable selection arguments within the proof of Proposition~\ref{prop:Psi-MultiP-OneP}, see also Remark~\ref{rem:normal}. 
\begin{lemma}\label{le:normal}
Let $\cD\subseteq \R^m$ be a set satisfying the grid-condition \eqref{eq:grid-condition}, 
 $f:\Omega^n\times\R^{m}\to \R\cup\{-\infty\}$ be a function  such that $(\omega,x)\mapsto f(\omega,x)$ is universally measurable, and for all $\omega\in \Omega^n$
 $\dom f(\omega,\cdot)\subseteq \cD$. Then $f$ is a $\cF_n$-normal integrand. In particular, the map $x\mapsto f(\omega,x)$ is upper-semicontinuous for every $\omega\in \Omega^n$. 
\end{lemma}
\begin{proof}
By the definition of being a $\cF_n$-normal integrand, we need to show that its $\mbox{hypo} f(\omega)$ is closed-valued and $\cF_n$-measurable in the sense of \cite[Definition~14.1,p.643]{RockafellarWets}.

To show that the $\mbox{hypo} f(\omega)$ is closed-valued, fix any $\omega \in \Omega^n$  and let $(x^k,\alpha^k)\subseteq \R^{m}\times \R$ converging to some $(x,\alpha)$ satisfying for each $k$ that $f(\omega,x^k)\geq \alpha^k$. We need to show that $f(\omega,x)\geq \alpha$. Observe that by the assumption $\dom f(\omega,\cdot)\subseteq \cD$, we have $(x^k)\subseteq \cD$. Therefore, by the grid-condition \eqref{eq:grid-condition} of the set $\cD$, the sequence $(x^k)$ is constant equal to $x$ for eventually all $k$. In particular, for large enough $k$, we have $f(\omega,x)\geq \alpha^k$, which then implies the desired inequality $f(\omega,x)\geq \alpha$.

To see that the $\mbox{hypo} f(\omega)$ is a $\cF_n$-measurable correspondence, notice first that by the assumption $\dom f(\omega,\cdot)\subseteq \cD$ and $\cD$ satisfies the grid-condition \eqref{eq:grid-condition}, the map $x\mapsto f(\omega,x)$ is upper-semicontinuous for every $\omega$. Hence, we deduce from \cite[Proposition~14.40, p.667]{RockafellarWets} that it suffices to show that for any $K\subset \R^m$ being compact, the function
\begin{equation*}
g(\omega):=\sup_{x \in K} f(\omega,x)
\end{equation*}
is $\cF_n$-measurable. To that end, fix any $K\subseteq \R^m$ being compact and then fix any constant $c \in \R$. We need to show that the set $\{\omega\,| \, g(\omega)> c\} \in \cF_n$. As $\dom f(\omega,\cdot)\subseteq \cD$, this is trivially satisfied if $\cD \cap K =\emptyset$. Therefore, from now on, assume that
$\cD \cap K \neq\emptyset$. By compactness of $K$ and the grid-condition \eqref{eq:grid-condition} of $\cD$, we get that $\cD \cap K=\{k_1,...,k_j\}$ is finite. Therefore, by definition of $g(\omega)$, we have
\begin{align*}
\{\omega\,| \, g(\omega)> c\}= \bigcup_{i=1}^j \{\omega\,|f(\omega,k_i)>c\},
\end{align*}
which is $\cF_n$-measurable by \cite[Lemma~7.29, p.174]{BertsekasShreve.78}.
\end{proof}

Before we can start 
with the proof of Proposition~\ref{prop:local-NA}, we need to see that the set valued map $K_t(\omega^t)$ has some desirable properties.
\begin{lemma}\label{le:K-t-local-nice}
  Let $t \in \{0,\dots,T-1\}$. Assume that $\Psi_{t+1}$ satisfy Assumption~\ref{ass:Psi-MultiP}. Then the set-valued map $K_t$ defined in \eqref{eq:def-K-t-local} is a closed-valued,
   $\cF_t$-measurable correspondence and the set $\big\{\omega^t \in \Omega^t \, \big|\, K_t(\omega^t)=\{0\} \big\}\in\cF_t$.
\end{lemma}
\begin{proof}
  As $\Phi_{t}^\infty$
  is positively homogeneous and upper-semicontinuous in $x^{t+1}$, $K_t(\omega^t)$ is a closed-valued cone for every $\omega^t$. We emphasize that as $\Psi$ is not assumed to be concave, $K_t(\omega^t)$ is not necessarily convex.
  %
  %
By Lemma~\ref{le:normal} and Proposition~\ref{Psi-DPP-Ass-Multi-P-OK}, 
   $\Phi_t$ is a $\cF_t$-normal integrand, hence so is $\Phi_t^\infty$.
  Thus, the set-valued map $K_t$ is an $\cF_t$-measurable correspondence; see  \cite[Proposition~14.33, p.663]{RockafellarWets} and \cite[Proposition~14.45(a), p.669]{RockafellarWets}. 
  By \cite[Theorem~14.5(a), p.646]{RockafellarWets}, $K_t$ admits a Castaing representation $\{x_n\}$.
Thus, 
  \begin{align}\label{Castaing-Report1}
\big\{\omega^t \in \Omega^t \, \big|\, K_t(\omega^t)=\{0\} \big\} =  \bigcap_{n \in \N} \{\omega^t \in \Omega^t\,|\, x_n(\omega^t)=0\} \in \cF_t.
  \end{align}
\end{proof}
\begin{lemma}\label{le:-local-K-global-K-relation}
Let $t\in\{0,\ldots,T-1\}$. Assume that $\Psi_{t+1}$ satisfies Assumption~\ref{ass:Psi-MultiP}. Moreover, let $X:\Omega^t \to \R^d$ be any $\cF_t$-measurable random variable. Then the following holds true.
\begin{equation*}
X(\omega^t) \in K_t(\omega^t) \ \ \fP\mbox{-q.a. }\omega^t \in \Omega^t   \quad \Longrightarrow  \quad (0,\dots,0,X)\in \cK^{t+1}
\end{equation*}
\end{lemma}
\begin{proof}
Let $X \in K_t \ \fP$-q.s. Then for any arbitrary $P \in \fP$, recall that its restriction $P^{t+1}$ to $\Omega^{t+1}$ is of the form $P^t\otimes P_t$ for some selector $P_t \in \fP_t$. Therefore, by Lemma~\ref{le:ass-psi-phi-multiP}
\begin{align*}
E^P[\Psi^\infty_{t+1}(0,\dots,0,X)]&=E^{P^t(d\omega^t)}\Big[
E^{P_t(\omega^t)}\big[\Psi^\infty_{t+1}\big(\omega^t\otimes_t\cdot,0,\dots,0,X(\omega^t)\big)\big]\Big]\\
&\geq E^{P^t(d\omega^t)}\big[
\Phi^\infty_{t}\big(\omega^t,0,\dots,0,X(\omega^t)\big)\big]\\
&\geq 0.
\end{align*}
By the arbitrariness of $P \in \fP$, we conclude that $(0,\dots,0,X)\in \cK^{t+1}$.
\end{proof}
%
%
%
%
%
%
Now we are able to prove Proposition~\ref{prop:local-NA}.
\begin{proof}[Proof of Proposition~\ref{prop:local-NA}]
Recall from the proof of Lemma~\ref{le:K-t-local-nice} that $K_t$ admits a Castaing representation $\{x_n\}$. Then using Lemma~\ref{le:-local-K-global-K-relation} and \eqref{Castaing-Report1},  we have that
  \begin{align*}
  & \ \mbox{$\mbox{NA}(\fP)^{t+1}$ up to time $t+1$ holds true}\\
 \Longrightarrow & \ \mbox{ $\forall n \in \N$: } (0,\dots,0,x_n)=0 \ \ \fP\mbox{-q.s.}\\
           \Longleftrightarrow &\ \mbox{ the set }
              \bigcap_{n \in \N} \{\omega^t \in \Omega^t\,|\, x_n(\omega^t)=0\} \mbox{ has $\fP$-quasi full measure}\\
               \Longleftrightarrow &\ \mbox{ $\mbox{NA}_t$ holds true.}
  \end{align*}
\end{proof}
%
%
\begin{proof}[Proof of Proposition~\ref{prop:Psi-MultiP-OneP}]
By Lemma~\ref{le:normal}, we have that $\Phi_t$ is an $\cF_t$-normal integrand. 
  	Then, for any fixed strategy $H^t \in \cH^t$, \cite[Proposition~14.45(c), p.669]{RockafellarWets} gives that the mapping $\Phi^{H^t}(\omega^t,x):=\Phi_t(\omega^t,H^t(\omega^{t-1}),x)$ is a $\cF_t$-normal integrand, too. Therefore, we deduce from \cite[Theorem~14.37, p.664]{RockafellarWets} that the set-valued mapping $\Upsilon\colon\Omega^t \rightrightarrows \R^d$ defined by
  	\begin{equation*}
  	\Upsilon(\omega^t):=\mathrm{argmax}\ \Phi^{H^t}(\omega^t,\cdot)
  	\end{equation*}
  	admits an $\cF_t$-measurable selector $\widehat{h}_t$ on the  $\cF_t$-measurable set $\{\Upsilon\neq \emptyset\}$. Extend $\widehat{h}_t$ by setting $\widehat{h}_t=0$ on $\{\Upsilon=\emptyset\}$. It remains to argue that $\{\Upsilon=\emptyset\}$ is a $\fP$-polar set.
  	
  	 To that end, we want to show that the map $(x^t,x_t) \mapsto \Phi_t(\omega^t,x^{t},x_t)$ satisfies the conditions of Proposition~\ref{prop:key-argument-RW}. By Lemma~\ref{Psi-DPP-Ass-Multi-P-OK}, it is proper and upper-semicontinuous; see also Remark~\ref{rem:Psi-usc}.
  	 Moreover, as the local no-arbitrage condition $\mbox{NA}_t$ holds true by Proposition~\ref{prop:local-NA}, we have for $\fP$-q.a. $\omega^t \in \Omega^t$ that
  	 \begin{align*}
  	\cK(\omega^t) :=\big\{h \in \R^d\,\big|\, \Phi_t^\infty(\omega^t,0,\dots,0,h)\geq 0\big\}
  	= K_t(\omega^t)
  	=\{0\}. 
  	 \end{align*}
  	  Hence the conditions of Proposition~\ref{prop:key-argument-RW} are indeed satisfied. 
  	Therefore, we conclude from Proposition~\ref{prop:key-argument-RW}, as by assumption $\sup_{x\in \R^d}\Phi_t(\omega^t,H^t(\omega^{t-1}),x)>-\infty \ \fP\mbox{-q.s.}$, that  $\{\Upsilon=\emptyset\}$ is a $\fP$-polar set.  The  result now follows.
\end{proof}
%
%
Now we continue with the proof of Proposition~\ref{prop:no-arbitrage-up-time-t-MultiP}.
%
%
%
%
\begin{proof}[Proof of Proposition~\ref{prop:no-arbitrage-up-time-t-MultiP}]
Assume by contradiction that $\cK^{t}\neq\{0\}$. Then, there exists a probability measure in $P \in \fP$ 
 and $\widetilde{H}^t\in \cH^t$ such that
\begin{equation}\label{Step1-1}
\Psi^{\infty}_t(\widetilde H^t)\geq 0 \quad\fP\mbox{-q.s.} \quad \ \mbox{ and } \ \quad P[\widetilde H^t\neq 0]>0.
\end{equation}
\textsc{Step 1:} We claim that there exists an $\cF_t$-measurable map $\widetilde{h}_t:\Omega^t \to \R^d$ such that
\begin{equation}\label{Step1-2}
\Phi^\infty_{t}(\widetilde{H}^t,\widetilde{h}_t)= \Psi^\infty_{t}(\widetilde{H}^t)\quad \fP\mbox{-q.s.}
\end{equation}
Indeed, first notice 
 that 
 $\Phi_t^\infty(0)\geq 0$ $\fP\mbox{-q.s.}$, in particular $\Phi_t^\infty$ is proper $\fP\mbox{-q.s.}$.  To see this, observe that by Lemma~\ref{le:ass-psi-phi-multiP}, $\Phi_t$ satisfies Assumption~\ref{ass:Psi-MultiP}(4), hence $\Phi_t(0)>-\infty$ $\fP\mbox{-q.s.}$. Therefore, we obtain  $\fP\mbox{-q.s.}$  that
 \begin{equation*}
 \Phi_t^\infty(0)= \lim_{n\to \infty} \sup_{\delta>n, \ |y|<\frac{1}{n}} \frac{1}{\delta} \Phi_t(\delta y)
 \geq
 \lim_{n\to \infty} \sup_{\delta>n} \frac{1}{\delta} \Phi_t(0)=0. 
 \end{equation*}
 
 Moreover, observe that $(\Phi_t^\infty)^\infty=\Phi_t^\infty$ by the property of a horizon function to be upper-semicontinuous and positively homogeneous; see \cite[p.87]{RockafellarWets}. Hence, as by Assumption $\mbox{NA}(\fP)^{t+1}$ up to time $t+1$ holds,  we know from Proposition~\ref{prop:local-NA} that 
the local no-arbitrage condition $\NA_t$ holds, too. Therefore, by the same arguments as in the proof of Proposition~\ref{prop:Psi-MultiP-OneP}, 
 the map $(x^t,x_t) \mapsto \Phi^\infty_t(x^t,x_t)$ satisfies the condition of Proposition~\ref{prop:key-argument-RW} $\fP\mbox{-q.s.}$. Since by assumption $\Psi^{\infty}_t(\widetilde H^t)\geq 0 \quad\fP\mbox{-q.s.}$, we conclude from Proposition~\ref{prop:key-argument-RW} 
  that the set-valued map
\begin{align*}
M(\omega^t):&=\big\{h \in \R^d \,\big| \,\Phi_{t}^\infty(\omega^t, \widetilde H^t,h)=\Psi_t^\infty(\omega^t, \widetilde H^t)\big\}\\
&=\big\{h \in \R^d \,\big| \,\Phi_{t}^\infty(\omega^t, \widetilde H^t,h)=\sup_{x\in \R^d} \Phi_{t}^\infty(\omega^t, \widetilde H^t,x)\big\}
\end{align*}
is not empty for $\fP$-q.e. $\omega^t$. 
Moreover, by Lemma~\ref{le:normal} and \cite[Exercise~14.54(a)p.673]{RockafellarWets}, the function $\Phi^\infty_t$ is a normal integrand. So, \cite[Theorem~14.37, p.664]{RockafellarWets} provides the existence of an $\cF_t$-measurable selector $\widetilde{h}_t$ of $M$.

\textsc{Step 2:} Let us show that $\widetilde{H}^{t+1}:=(\widetilde H^t,\widetilde h_t) \in \cH^{t+1}$ satisfies $\Psi^{\infty}_{t+1}(\widetilde H^{t+1})\geq 0 \ \fP\mbox{-q.s.}$, i.e. $\widetilde{H}^{t+1} \in \cK^{t+1}$.

For $\fP$-q.e. $\omega^t$ we have by \eqref{Step1-1}, \eqref{Step1-2}, and Lemma~\ref{le:ass-psi-phi-multiP} 
\begin{align*}
0= \Psi^{\infty}_t(\omega^t,\widetilde H^t(\omega^{t-1})) =\Phi^{\infty}_{t}(\omega^t,\widetilde H^{t+1}(\omega^t))\leq\inf_{P \in \fP_t(\omega^t)} E^P[\Psi^\infty_{t+1}(\omega^t\otimes_t \cdot,\widetilde H^{t+1}(\omega^t))]\leq 0,
\end{align*}
which in turn implies that for $\fP$-q.e. $\omega^t$ we have that
\begin{equation*}
0= \inf_{P \in \fP_t(\omega^t)} E^P[\Psi^\infty_{t+1}(\omega^t\otimes_t \cdot,\widetilde H^{t+1}(\omega^t))]
\end{equation*}
As every $P'\in\fP$ satisfies $P'|_{\Omega^{t+1}}=P'|_{\Omega^{t}}\otimes P'_t$ for some selection $P'_t\in\fP_t$, we obtain the result directly from Fubini's theorem.

But now, we have constructed $\widetilde{H}^{t+1} \in \cK^{t+1}$ with
\begin{equation*}
P[\widetilde{H}^{t+1}\neq 0]\geq P[\widetilde{H}^{t}\neq 0]>0,
\end{equation*}
which is a contradiction to $\mbox{NA}(\fP)^{t+1}$ up to time $t+1$. Thus, we must have $\cK^{t}=\{0\}$.
\end{proof}
%
%
%
%
\subsection{Proof of Theorem~\ref{thm:Maxim-exist-MultiP}}\label{subsec:Proof-main-thm}
The goal of this subsection is to give the proof of Theorem~\ref{thm:Maxim-exist-MultiP}, which is the main result of this paper. We will construct the optimal strategy $\widehat{H}:=(\widehat H_0,\dots,\widehat H_{T-1}) \in \cH$ recursively from time $t=0$ upwards by applying Proposition~\ref{prop:Psi-MultiP-OneP} at each time $t$, given the  restricted strategy $\widehat{H}^t:=(\widehat H_0,\dots,\widehat H_{t-1})\in \cH^t$. We follow \cite{Nutz.13util} to check that $\widehat{H}$ is indeed an optimizer of \eqref{eq:thm-optimal-MultiP}.
%
%
%
%
%
\begin{proof}[Proof of Theorem~\ref{thm:Maxim-exist-MultiP}]
First, let us check that the conditions in Theorem~\ref{thm:1-Per} are satisfied for the function $\Psi_1$. Indeed, by
Proposition~\ref{Psi-DPP-Ass-Multi-P-OK}  and Proposition~\ref{prop:no-arbitrage-up-time-t-MultiP} we know that  the no-arbitrage condition \eqref{eq:def-NA-1-Per} in the one-period-model holds true. 
Moreover, 
by Proposition~\ref{Psi-DPP-Ass-Multi-P-OK} we know that $\Psi_1$ satisfies Assumption~\ref{ass:Psi-MultiP}.
 This together with Lemma~\ref{le:normal} shows that $\Psi_1$ satisfies Assumption~\ref{ass:1-Per}. 
Hence by Theorem~\ref{thm:1-Per},
there exists $\widehat{H}_0\in \R^d$ such that
\begin{equation*}
\inf_{P_\in \fP_0} E^P[\Psi_1(\widehat{H}_0)]=\sup_{x\in \R^d} \inf_{P_\in \fP_0} E^P[\Psi_1(x)]>-\infty.
\end{equation*}
Now, we claim that for each $t\in \{0,\dots,T-1\}$ we have that
\begin{equation}\label{eq:induction}
\sup_{x\in \R^d}\Phi_t(\omega^t,\widehat H^t(\omega^{t-1}),x)>-\infty \quad \fP\mbox{-q.s.,}
\end{equation}
where $\widehat{H}^t:=(\widehat H_0,\dots,\widehat H_{t-1})$ is recursively defined from time $t=0$ upwards. We argue by induction. Indeed, the base case $t=0$ we just argued above.  For the induction step, assume now that \eqref{eq:induction} holds true at time $t-1$. Then, by Proposition~\ref{prop:Psi-MultiP-OneP} there exists an $\mathcal{F}_{t-1}$-measurable random variable $\widehat{H}_{t-1}$ such that for $\fP$-q.e.\ $\omega^{t-1}$ it holds that
\begin{equation}\label{eq:induction-1}
-\infty<\sup_{x\in \R^d}\Phi_{t-1}(\omega^{t-1},\widehat H^{t-1}(\omega^{t-2}),x)= \Phi_{t-1}(\omega^{t-1},\widehat H^{t-1}(\omega^{t-2}),\widehat{H}_{t-1}(\omega^{t-1})).
\end{equation}
We write $\widehat{H}^t:=(\widehat H^{t-1},\widehat{H}_{t-1})\in  \mathcal{H}^t$. Notice that by definition of the recursion defined in \eqref{eq:Multi-P-Recursion}, we have for every $\omega^{t-1} \in \Omega^{t-1}$ that
\begin{equation*}
\Phi_{t-1}(\omega^{t-1},\widehat H^{t}(\omega^{t-1}))
=
\inf_{P \in \fP_{t-1}(\omega^{t-1})}E^P[\Psi_{t}(\omega^{t-1}\otimes_{t-1} \cdot,\widehat H^t(\omega^{t-1}))].
\end{equation*}
This and \eqref{eq:induction-1} ensure that 
\begin{equation}\label{eq:induction-2}
\mbox{for $\fP$-q.e.\ $\omega^{t-1}$ we have }\ \Psi_{t}(\omega^{t-1}\otimes_{t-1} \cdot,\widehat H^t(\omega^{t-1}))>-\infty \quad \fP_{t-1}(\omega^{t-1})\mbox{-q.s..}
\end{equation}
In addition, recall for every $P \in \fP$ that its restriction $P^{t}$ to $\Omega^t$ is of the form $P^{t-1}\otimes P_{t-1}$ for some selector $P_{t-1}\in \fP_{t-1}$. Therefore, \eqref{eq:induction-2} implies that for every $P \in \fP$
\begin{align*}
P\!\big[\Psi_{t}(\widehat H^t)>-\infty\big]
&=
E^{P^{t-1}(d\omega^{t-1})}\Big[E^{P_{t-1}(\omega^{t-1})}
\big[\mathbf{1}_{\{\Psi_{t}(\omega^{t-1}\otimes_{t-1} \cdot,\widehat H^t(\omega^{t-1}))>-\infty\}}\big]\Big]=1.
\end{align*}
Hence by definition of the recursion \eqref{eq:Multi-P-Recursion}, we see that indeed
\begin{equation*}
\sup_{x\in \R^d}\Phi_t(\omega^t,\widehat H^t(\omega^{t-1}),x)
=\Psi_{t}(\omega^t,\widehat H^t(\omega^{t-1}))
>-\infty \quad \mbox{for }\,\fP\mbox{-q.e. $\omega^t$.}
\end{equation*}
As a consequence, we can apply Proposition~\ref{prop:Psi-MultiP-OneP} to 
 find an $\cF_t$-measurable random variable $\widehat{H}_t$ such that
\begin{equation*}
\inf_{P \in \fP_t(\omega^t)}E^P[\Psi_{t+1}(\omega^t\otimes_t \cdot,\widehat H^t(\omega^{t-1}),\widehat{H}_{t}(\omega^t))]
=\Psi_t(\omega^t,\widehat H^t(\omega^{t-1}))
\end{equation*}
for $\fP$-quasi-every $\omega^t\in \Omega^t$, for all $t=1,\dots T-1$. We claim that $\widehat H\in\cH$ is optimal, i.e. satisfies \eqref{eq:thm-optimal-MultiP}.
We first show that
\begin{equation}\label{eq:pf-thm-claim1-MultiP}
\inf_{P\in \fP}E^P[\Psi_T(\widehat H)]\geq \Psi_0.
\end{equation}
To that end, let $t \in \{0,\dots,T-1\}$. Let $P \in \fP$; we write $P= P_0\otimes\dots\otimes P_{T-1}$ with kernels $\fP_s:\Omega^s\to\fM_1(\Omega_1)$  satisfying $P_s(\cdot)\in \fP_s(\cdot)$. Therefore, by applying Fubini's theorem and the definition of $\widehat H$
\begin{align*}
 \ E^P[&\Psi_{t+1}(\widehat{H}_0,\dots,\widehat{H}_{t})]\\[1ex]
%
%
&\geq  \ E^{(P_0\otimes\dots\otimes P_{t-1})(d\omega^t)}\Big[\inf_{P' \in \fP_t(\omega^t)}E^{P'}\big[\Psi_{t+1}(\omega^t\otimes_t\,\cdot,\widehat H^t(\omega^{t-1}),\widehat{H}_{t}(\omega^t))\big]\Big]\\[1ex]
&= \ E^{(P_0\otimes\dots\otimes P_{t-1})}[\Psi_t(\widehat H^t)]
= E^P[\Psi_t(\widehat H^t)].
\end{align*}
Using this inequality repeatedly from $t=T-1$ to $t=0$ yields $E^P[\Psi_T(\widehat H)]\geq \Psi_0$. As $P \in \fP$ was arbitrarily chosen, the claim \eqref{eq:pf-thm-claim1-MultiP} is proven. It remains to show that
\begin{equation*}
\Psi_0 \geq \sup_{H \in \cH}\inf_{P \in \fP}E^P[\Psi(H)] 
\end{equation*}
to see that $\widehat{H} \in \cH$ is optimal. So, fix an arbitrary $H \in \cH$. It suffices to show that for every $t \in \{0,\dots,T-1\}$
\begin{equation}\label{eq:pf-thm-claim2-MultiP}
\inf_{P \in \fP}E^P[\Psi_t(H^t)]
\geq
\inf_{P \in \fP}E^P[\Psi_{t+1}(H^{t+1})].
\end{equation}
Indeed, by using the inequality repeatedly from $t=0$ until $t=T-1$  we get that 
 $\Psi_0\geq \inf_{P\in \fP}E^{P}[\Psi_{T}( H)]$.
Furthermore, as $H \in \cH$ was arbitrary and $\Psi_T=\Psi$, we obtain the desired inequality.

Now, to prove the inequality in~\eqref{eq:pf-thm-claim2-MultiP}, fix an $\varepsilon>0$. By \cite[Lemma~5.13]{NeufeldSikic.16}, see also \cite[Proposition~7.50]{BertsekasShreve.78}, there exists a kernel $P^\varepsilon_t:\Omega^t\to\fM_1(\Omega_1)$ such that for all $\omega^t \in \Omega^t$
\begin{align*}
 & E^{P^\varepsilon_t(\omega^t)}[\Psi_{t+1}(\omega^t\otimes_t\,\cdot, H^{t+1}(\omega^t))]-\varepsilon\\[1ex]
&\leq  \ (-\varepsilon^{-1})\vee \inf_{P \in \fP_t(\omega^t)}E^P[\Psi_{t+1}(\omega^t\otimes_t\,\cdot, H^{t+1}(\omega^t))]\\[1ex]
&\leq  \ (-\varepsilon^{-1})\vee \sup_{x \in \R^d}\inf_{P \in \fP_t(\omega^t)}E^P[\Psi_{t+1}(\omega^t\otimes_t\,\cdot, H^t(\omega^{t-1}),x)] 
= \ (-\varepsilon^{-1})\vee \Psi_t(\omega^t,H^t(\omega^{t-1})).
\end{align*}
Take any $P \in \fP$ and denote its restriction to $\Omega^t$ by $P^t$. Integrating the above inequalities yields
\begin{align*}
E^{P^t}[(-\varepsilon^{-1})\vee \Psi_t(H^t)]
\geq
E^{P^t\otimes P^\varepsilon_t}[\Psi_{t+1}( H^{t+1})]-\varepsilon 
\geq
\inf_{P'\in \fP}E^{P'}[\Psi_{t+1}( H^{t+1})]-\varepsilon.
\end{align*}
Letting $\varepsilon \to 0$, we obtain, by Fatou's Lemma, that
\begin{equation*}
E^P[ \Psi_t(H^t)] \geq \inf_{P'\in \fP}E^{P'}[\Psi_{t+1}( H^{t+1})].
\end{equation*}
\end{proof}
%
%

%
%
%

%
%

\newcommand{\dummy}[1]{}


\end{document}